\documentclass[11pt,oneside,leqno]{amsart}
\usepackage{amsxtra}
\usepackage{amsopn}
\usepackage{color}
\usepackage{amsmath,amsthm,amssymb}
\usepackage{amscd}
\usepackage{amsfonts}
\usepackage{latexsym}
\usepackage{verbatim}
\usepackage{multirow}

\theoremstyle{plain}
\newtheorem{theorem}{Theorem}[section]

\newtheorem{proposition}[theorem]{Proposition}
\newtheorem{corollary}[theorem]{Corollary}
\newtheorem{remark}[theorem]{Remark}

\newtheorem{remark-question}[section]{Remark-Question}

\newcommand\frg{{\mathfrak g}}

\newcommand\db{{\bar{\partial}}}

\begin{document}

\title[]{On non-K\"ahler compact complex manifolds with balanced 
and astheno-K\"ahler metrics}


\author{Adela Latorre \hskip.5cm }
\address[]
{Departamento de Matem\'aticas\,-\,I.U.M.A.\\
Universidad de Zaragoza\\
Campus Plaza San Francisco\\
50009 Zaragoza, Spain}
\email{adela@unizar.es, ugarte@unizar.es}

\author{\ Luis Ugarte}


\maketitle

\begin{abstract}
In this note we construct, for every $n \geq 4$, a non-K\"ahler compact complex manifold $X$
of complex dimension $n$ admitting a balanced metric and an astheno-K\"ahler metric which is in addition
$k$-th Gauduchon for any $1\leq k\leq n-1$.
\end{abstract}

\section{Introduction}\label{introduction}

\noindent
Let $X$ be a compact complex manifold of complex dimension $n$, and let $F$ be a Hermitian metric on $X$.
It is well-known that the metric $F$ is called \emph{balanced} if the Lee form vanishes,
equivalently the form $F^{n-1}$ is closed.
If $\partial\db F^{n-2}=0$, then the Hermitian
metric $F$ is said to be \emph{astheno-K\"ahler}. Balanced metrics are studied by Michelsohn in \cite{Mi},
and the class of astheno-K\"ahler metrics is considered by Jost and Yau in \cite{JY}
to extend Siu's rigidity theorem to non-K\"ahler manifolds.
This note is motivated by a question in the paper \cite{STW} by Sz\'ekelyhidi, Tosatti and Weinkove
about the existence of examples of non-K\"ahler compact complex manifolds
admitting both balanced and astheno-K\"ahler metrics.
Recently, two examples, in dimensions 4 and 11, are constructed by Fino, Grantcharov and Vezzoni in \cite{FGV}.
Our goal is to present examples
in any complex dimension $n \geq 4$.
Moreover, we show that our astheno-K\"ahler metrics satisfy
the stronger condition of being $k$-th Gauduchon for every $1\leq k\leq n-1$.

When the Lee form is co-closed, equivalently $F^{n-1}$ is
$\partial \overline\partial$-closed, the Hermitian metric $F$ is
called \emph{standard} or \emph{Gauduchon}. By \cite{Gau}
there is a Gauduchon metric in the conformal class of every Hermitian metric on $X$.
Fu, Wang and Wu introduce and study in \cite{FWW} the following generalization of Gauduchon metrics.
Let $k$  be an integer such that $1 \leq k \leq n-1$, a Hermitian metric $F$ on $X$ is called {\em $k$-th Gauduchon} if
$\partial \overline \partial F^k  \wedge F^{n -k-1} =0$.

By definition, $(n-1)$-th Gauduchon metrics are the usual Gauduchon metrics.
Astheno-K\"ahler metrics are particular examples of $(n-2)$-th Gauduchon metrics, and any \emph{pluriclosed} (\emph{SKT}) metric,
i.e. a metric satisfying  $\partial\db F=0$,
is in particular $1$-st Gauduchon.

In \cite{FWW} a unique constant $\gamma_k (F)$ is associated to any Hermitian metric $F$ on $X$. This
constant is invariant by biholomorphisms and depends smoothly on $F$. Moreover, it is proved that
$\gamma_k (F) =0$ if and only if there exists a $k$-th Gauduchon metric in the conformal class of $F$.

On a compact complex surface any Hermitian metric is automatically astheno-K\"ahler, and the balanced condition is the same as
the K\"ahler one. In
complex dimension $n=3$ the notion of astheno-K\"ahler metric
coincides with that of SKT metric.

SKT or astheno-K\"ahler metrics on a compact complex manifold $X$ of complex dimension~$n \geq 3$ cannot be balanced
unless they are
K\"ahler (see \cite{AI,MT}). If the Lee form is exact, then the Hermitian structure is conformally balanced.
By \cite{FT,IP01} a conformally balanced SKT or astheno-K\"ahler metric
whose Bismut connection has (restricted)
holonomy contained in $SU(n)$ is necessarily K\"ahler.
Similar results for $1$-st Gauduchon metrics are proved in \cite{FU}.
Ivanov and Papadopoulos \cite{IP} have extended these results to any generalized $k$-th Gauduchon metric, for $k \not=n-1$.

A recent conjecture in \cite{FV1}
asserts that if $X$ has an SKT metric and another metric which is balanced, then $X$ is K\"ahler.
By a result of Chiose \cite{Chiose}
a manifold in the Fujiki class~$\mathcal{C}$ has no SKT metrics unless it is K\"ahler.
In \cite{FV2} the conjecture is studied on the class of \emph{complex nilmanifolds} $X=(\Gamma \backslash G, J)$,
i.e. on compact quotients
of simply-connected nilpotent Lie groups $G$ by uniform discrete subgroups $\Gamma$
endowed with an invariant complex structure $J$.
In this note we construct, for every $n \geq 4$, a non-SKT complex nilmanifold $X$
of complex dimension $n$ admitting a balanced metric and an astheno-K\"ahler metric which additionally satisfies
the stronger condition of being $k$-th Gauduchon for every $1\leq k\leq n-1$.

\section{Generalized Gauduchon metrics on complex nilmanifolds}\label{k-Gauduchon-sec}

\noindent We first prove the following general result.

\begin{proposition}\label{k-Gauduchon}
Let $X$ be a compact complex manifold of complex dimension $n\geq 3$,
and~$F$ any Hermitian metric on $X$.
For any integer $k$ such that $1 \leq k \leq n-1$, we have
\begin{equation}\label{formula1}
\int_X \partial\bar{\partial}F^k\wedge F^{n-k-1}
= \frac{k(n-k-1)}{n-2}\,\int_X \partial\bar{\partial}F\wedge F^{n-2}.
\end{equation}
\end{proposition}

\begin{proof}
The equality \eqref{formula1} is trivial for $k=1$ and for $k=n-1$. Let us then suppose that $2 \leq k \leq n-2$.
By induction
one has $\partial F^k=k\,\partial F\wedge F^{k-1}$ and
$\bar{\partial} F^k=k\,\bar{\partial} F\wedge F^{k-1}$. Therefore,
\begin{equation}\label{triangulo}
\partial\bar{\partial} F^k\wedge F^{n-k-1}=
k\,\partial\bar{\partial} F\wedge F^{n-2}+k(k-1)\,\partial F\wedge\bar{\partial} F\wedge F^{n-3}.
\end{equation}

On the other hand, 
\begin{eqnarray*}
\partial\bar{\partial} F^k\wedge F^{n-k-1}
\!&\!=\!&\!
d\left(\bar{\partial} F^k\wedge F^{n-k-1}\right)+\bar{\partial} F^k\wedge\partial F^{n-k-1} \\
\!&\!=\!&\! d\left(\bar{\partial} F^k\wedge F^{n-k-1}\right)-k\,(n-k-1)\partial F\wedge\bar{\partial} F\wedge F^{n-3},\nonumber
\end{eqnarray*}
so we get
$$
\partial F\wedge\bar{\partial} F\wedge F^{n-3}=\frac{-1}{k(n-k-1)}
\left[\partial\bar{\partial} F^k\wedge F^{n-k-1} - d\left(\bar{\partial} F^k\wedge F^{n-k-1}\right)\right].
$$
Now, if we substitute this expression in \eqref{triangulo} we have
\begin{eqnarray*}
\partial\bar{\partial} F^k\wedge F^{n-k-1}=k\,\partial\bar{\partial} F\wedge F^{n-2}
- \frac{k-1}{n-k-1}\,\partial\bar{\partial} F^k\wedge F^{n-k-1}
+ \frac{k-1}{n-k-1}\,d\left(\bar{\partial} F^k\wedge F^{n-k-1}\right),
\end{eqnarray*}
which leads to
$$(n-2)\,\partial\bar{\partial} F^k\wedge F^{n-k-1} =
  k(n-k-1)\,\partial\bar{\partial} F\wedge F^{n-2} + (k-1)\, d\left(\bar{\partial} F^k\wedge F^{n-k-1}\right).$$
By Stokes theorem we arrive at \eqref{formula1}.
\end{proof}

Next we apply the previous proposition to homogeneous compact complex manifolds $X$, of complex dimension~$n$,
endowed with an \emph{invariant} Hermitian metric $F$.
We recall that in \cite[Lemma 4.7]{FU} the following duality result
is proved: for each $k=1,\ldots,\left[\frac{n}{2}\right]-1$,
the Hermitian metric $F$ is $k$-th Gauduchon if and only if it is $(n-k-1)$-th Gauduchon.
As a consequence of Proposition~\ref{k-Gauduchon}, the relation among these metrics turns out to be stronger:

\begin{proposition}\label{k-Gauduchon-nil}
Let $F$ be an invariant Hermitian
metric on a homogeneous compact complex manifold $X$ of complex dimension $n\geq 3$,
and let $k$ be an integer such that $1 \leq k \leq n-2$.
Then,
\begin{enumerate}
\item[(i)] $F$ is always Gauduchon, and
\item[(ii)] if $F$ is $k$-th Gauduchon for some $k$, then it is $k$-th Gauduchon for any other $k$.
\end{enumerate}
\end{proposition}

\begin{proof}
For any invariant Hermitian metric $F$ and any $1 \leq k \leq n-2$,
the real $(n,n)$-form $\frac{i}{2}\,\partial\bar{\partial}F^k\wedge F^{n-k-1}$ is proportional to the volume form $F^n$,
hence
\begin{equation}\label{formula-constante}
\frac{i}{2}\,\partial\bar{\partial}F^k\wedge F^{n-k-1}= C_{F,k}\, F^n,
\end{equation}
for some constant $C_{F,k} \in \mathbb{R}$ (notice that $C_{F,k}$ is a multiple of the constant $\gamma_k(F)$ in \cite{FWW}).

If $k=n-1$ then $C_{F,n-1}=0$, i.e. $F$ is Gauduchon, because otherwise the form $F^n$ would be exact.
Now, let $k$ be such that $1 \leq k \leq n-2$. From~\eqref{formula1} and~\eqref{formula-constante} we get
$$
C_{F,k} \int_X F^n
= \frac{i}{2} \int_X \partial\bar{\partial}F^k\wedge F^{n-k-1}
=\frac{k(n-k-1)}{n-2}\,\frac{i}{2} \int_X \partial\bar{\partial}F\wedge F^{n-2}
= \frac{k(n-k-1)}{n-2}\,C_{F,1} \int_X F^n,
$$
that is,
$
\left( C_{F,k} - \frac{k(n-k-1)}{n-2}\,C_{F,1} \right) \, \int_X F^n =0.
$
Therefore,
\begin{equation}\label{relacion}
C_{F,k} = \frac{k(n-k-1)}{n-2}\,C_{F,1}\,,
\end{equation}
for any $k$ such that $1 \leq k \leq n-2$. Hence, if $F$ is $k$-th Gauduchon for some $k$, then $C_{F,k}=0$ and by
\eqref{relacion} we get $C_{F,1}=0$. Using again \eqref{relacion} we conclude that $C_{F,k}=0$ for any other $k$, i.e.
$F$ is $k$-th Gauduchon for any $1 \leq k \leq n-2$.
\end{proof}

\begin{corollary}\label{AK is k-Gauduchon}
Let $X$ be a homogeneous compact complex manifold of complex dimension $n\geq 3$ and let $F$ be an invariant Hermitian metric on $X$.
If $F$ is SKT or astheno-K\"ahler, then $F$ is $k$-th Gauduchon for any $1 \leq k \leq n-1$.
\end{corollary}

\begin{theorem}\label{nonK-balanced-kG}
For each $n \geq 4$, there is a non-K\"ahler compact complex manifold $X$ of complex dimension $n$
admitting a balanced metric $\tilde{F}$
and an astheno-K\"ahler metric $F$
which is additionally $k$-th Gauduchon for any $1\leq k\leq n-1$.
\end{theorem}

\begin{proof}
We will construct such an $X$ using the class of complex nilmanifolds.
Let $(a_{1},\ldots,a_{n-1}) \in (\mathbb{R}\!\setminus\!\{0\})^{n-1}$, and
let $\{\omega^j\}_{j=1}^n$ be a basis of forms of type (1,0) satisfying
\begin{equation}\label{ecus}
d\omega^1=\cdots=d\omega^{n-1}=0,\quad
d\omega^{n}=\sum_{j=1}^{n-1} a_{j}\, \omega^{j\bar{j}}.
\end{equation}
(See Remark~\ref{AndradaBD} below for more details.)
We impose the ``canonical'' metric
$\tilde{F} = \frac{i}{2} (\omega^{1\bar{1}}+ \cdots+ \omega^{n\bar{n}})$
to be balanced, i.e. $d\tilde{F}^{n-1}=0$.
This condition is equivalent to
\begin{equation}\label{balanced-condition}
a_{1} + \cdots + a_{n-1}=0.
\end{equation}
Let us now consider a generic ``diagonal'' metric
\begin{equation}\label{diagonal-metric}
F= \frac{i}{2} (b_1\, \omega^{1\bar{1}}+ \cdots+ b_{n-1}\, \omega^{{n-1}\overline{n-1}})
+ \frac{i}{2} \, \omega^{n\bar{n}},
\end{equation}
where $b_1, \ldots, b_{n-1} \in \mathbb{R}^+$.

Let $r\leq n-1$. We denote
$A_{r}=a_{1}\, \omega^{1\bar{1}}+ \cdots +a_{r}\, \omega^{r\overline{r}}$
and $B_{r}=b_{1}\, \omega^{1\bar{1}}+ \cdots +b_{r}\, \omega^{r\overline{r}}$.
Hence, in~\eqref{ecus} and \eqref{diagonal-metric}
we can write $d\omega^{n}=A_{n-1}$ and $F=\frac{i}{2}\, B_{n-1} + \frac{i}{2}\, \omega^{n\bar{n}}$.

Let us calculate $\partial \db F^{n-2}$.
Using that the form $B_{n-1}$ is closed, we get
\begin{eqnarray*}
(-2i)^{n-2}\partial \db F^{n-2}
\!&\!=\!&\!
\partial \db \left(B_{n-1} + \omega^{n\bar{n}} \right)^{n-2}
=
\partial \db (B_{n-1})^{n-2} + (n-2)\, \partial \db \left( (B_{n-1})^{n-3} \wedge \omega^{n\bar{n}} \right) \\
\!&\!=\!&\!
(n-2)\, (B_{n-1})^{n-3} \wedge \partial \db(\omega^{n\bar{n}})
=-(n-2)\, (A_{n-1})^2 \wedge (B_{n-1})^{n-3},\nonumber
\end{eqnarray*}
where in the last equality we have used that
$\partial \db(\omega^{n\bar{n}})=\db \omega^n \wedge\partial \omega^{\bar{n}}=-A_{n-1} \wedge A_{n-1}$.

Therefore, $F$ is astheno-K\"ahler if and only if $(A_{n-1})^2 \wedge (B_{n-1})^{n-3}=0$.

We now use the balanced condition~\eqref{balanced-condition}, i.e. $a_{n-1}=-a_1- \cdots -a_{n-2}$.
Writing $A_{n-1} = A_{n-2} + a_{n-1}\, \omega^{n-1\overline{n-1}}$
and
$B_{n-1} = B_{n-2} + b_{n-1}\, \omega^{n-1\overline{n-1}}$, and
noting that $(A_{n-2})^2 \wedge (B_{n-2})^{n-3}=0$, one has that the astheno-K\"ahler condition
is equivalent to
\begin{eqnarray*}
0 \!&\!=\!&\! (A_{n-1})^2 \wedge (B_{n-1})^{n-3}
=
(A_{n-2} + a_{n-1}\, \omega^{n-1\overline{n-1}})^2 \wedge (B_{n-2} + b_{n-1}\, \omega^{n-1\overline{n-1}})^{n-3}\\[5pt]
\!&\!=\!&\! \left[ (A_{n-2})^2 + 2\, a_{n-1}\, A_{n-2} \wedge \omega^{n-1\overline{n-1}} \right]
\wedge \left[ (B_{n-2})^{n-3} + (n-3)\, b_{n-1}\, (B_{n-2})^{n-4} \wedge \omega^{n-1\overline{n-1}} \right]\\[5pt]
\!&\!=\!&\! \left[ (n-3)\, b_{n-1}\, A_{n-2} - 2\, (a_1+ \cdots + a_{n-2})\, B_{n-2} \right]
\wedge A_{n-2} \wedge (B_{n-2})^{n-4} \wedge \omega^{n-1\overline{n-1}}.\nonumber
\end{eqnarray*}

Let us observe that in order to simplify this equation one can take
$a_1,\ldots,a_{n-2}>0$ and $b_j=a_j$ for $1 \leq j \leq n-2$.
Indeed, in this case we have that $B_{n-2}=A_{n-2}$, so it is enough to choose $b_{n-1}=\frac{2}{n-3}(a_1+ \cdots + a_{n-2})$
to get an astheno--K\"ahler metric $F$ given by~\eqref{diagonal-metric}.

Finally, by Corollary~\ref{AK is k-Gauduchon} the metric $F$ is in addition $k$-th Gauduchon for any $1 \leq k \leq n-1$.
We notice that
it can be directly proved that these complex nilmanifolds do not admit any SKT metric.
Let us also note that the canonical bundle is holomorphically trivial, since the $(n,0)$-form
$\Omega=\omega^{1 \cdots n}$ is closed.
\end{proof}

\begin{remark}\label{AndradaBD}
{\rm
The (real) nilmanifolds in \eqref{ecus} correspond to the Lie algebras $\frg=\mathfrak{h}_{2n+1} \times \mathbb{R}$,
where $\mathfrak{h}_{2n+1}$ is the $(2n+1)$-dimensional Heisenberg algebra. Andrada, Barberis and Dotti proved in \cite[Proposition 2.2]{ABD}
that every invariant complex structure $J$ on these nilmanifolds is \emph{abelian}, i.e. $[Jx,Jy]=[x,y]$ for any $x,y\in\frg$. Moreover,
there are exactly $\left[\frac{n}{2}\right]+1$ complex structures up to isomorphism. Let $J_0$ be the complex structure
defined by taking all the coefficients $a_j$ positive numbers, i.e. $a_{1},\ldots,a_{n-1}>0$ in \eqref{ecus}.
One can prove the following result: \emph{for any $J$ not isomorphic to $J_0$, the complex nilmanifold admits
a balanced metric and an astheno-K\"ahler metric which is $k$-th Gauduchon for any $k$}.
}
\end{remark}

\begin{remark}\label{ejemplos-J-no-abeliana}
{\rm
The complex structure in the 4-dimensional example given in \cite{FGV} as well as those given in \eqref{ecus} are all abelian.
Here we present a more general family of 4-dimensional complex nilmanifolds where the complex structure is not of that special type.
Let us consider the complex structure equations
\begin{equation}\label{ecus-bis}
d\omega^1=d\omega^2=d\omega^{3}=0,\quad
d\omega^{4}=A\,\omega^{12} +B\,\omega^{13} +C\,\omega^{23} +\omega^{1\bar{1}}+\omega^{2\bar{2}}-2\,\omega^{3\bar{3}},
\end{equation}
where we require the coefficients $A,B,C$ to belong to $\mathbb{Q}(i)$ in order to ensure the existence of a lattice,
so that equations \eqref{ecus-bis} define a complex nilmanifold.
Consider a metric $F_{\alpha,\beta,\gamma}$ of the form
$$
F_{\alpha,\beta,\gamma} = \frac{i}{2} (\alpha\,\omega^{1\bar{1}}+ \beta\,\omega^{2\bar{2}}+\gamma\,\omega^{3\bar{3}}+ \omega^{4\bar{4}}),
$$
with $\alpha,\beta,\gamma \in \mathbb{R}^+$. On the one hand,
it is easy to see that $\alpha=\beta=\gamma=1$ provides a balanced metric.
On the other hand, the astheno-K\"ahler condition is satisfied
if and only if $\gamma=\frac{\alpha\,(|C|^2+4)+\beta\,(|B|^2+4)}{2-|A|^2} >0$,
so it suffices to take any complex structure in \eqref{ecus-bis} with $|A| < \sqrt{2}$.
This provides a family of 4-dimensional complex nilmanifolds $X_{A,B,C}$ with balanced and
astheno-K\"ahler metrics which are $k$-th Gauduchon for any $k$. Notice that if $(A,B,C) \not=(0,0,0)$, then the Lie algebra underlying
$X_{A,B,C}$ is not isomorphic to $\mathfrak{h}_{7} \times \mathbb{R}$.
}
\end{remark}

\medskip

\section*{Acknowledgments}
\noindent This work has been partially supported by the projects MINECO (Spain) MTM2014-58616-P
and Gobierno de Arag\'on/Fondo Social Europeo, grupo consolidado E15-Geometr\'{\i}a.
Adela Latorre is also supported by a DGA predoctoral scholarship.
We would like to thank Anna Fino for useful comments on the subject.
We also thank the referee for comments and suggestions
that have helped us to improve the final version of the paper.


\vspace{-0.25cm}


\begin{thebibliography}{33}

\bibitem{AI} B. Alexandrov, S. Ivanov, Vanishing theorems on Hermitian manifolds, \emph{Diff. Geom. Appl.} {\bf 14}
(2001), 251--265.

\bibitem{ABD} A. Andrada, M.L. Barberis, I. Dotti, Classification of abelian complex structures on 6-dimensional Lie algebras,
\emph{J. Lond. Math. Soc. (2)} \textbf{83} (2011), no.~1, 232--255;
Corrigendum: \emph{J. Lond. Math. Soc. (2)} \textbf{87} (2013), no.~1, 319--320.

\bibitem{Chiose} I. Chiose,
Obstructions to the existence of K\"ahler structures on compact complex manifolds,
\emph{Proc. Amer. Math. Soc.} \textbf{142} (2014), 3561--3568.

\bibitem{FGV} A. Fino, G. Grantcharov, L. Vezzoni, Astheno-K\"ahler and balanced structures
on fibrations, arXiv:1608.06743 [math.DG].

\bibitem{FT} A. Fino, A. Tomassini, On astheno-K\"ahler metrics,
\emph{J. London Math. Soc.} {\bf 83} (2011), 290--308.

\bibitem{FU} A. Fino, L. Ugarte, On generalized Gauduchon metrics, \emph{Proc. Edinburgh Math. Soc.}
{\bf 56} (2013), 733--753.

\bibitem{FV1} A. Fino, L. Vezzoni, Special Hermitian metrics on compact solvmanifolds,
\emph{J. Geom. Phys.} \textbf{91} (2015), 40--53.

\bibitem{FV2} A. Fino, L. Vezzoni, On the existence of balanced and SKT metrics on nilmanifolds,
\emph{Proc. Amer. Math. Soc.} \textbf{144} (2016), 2455--2459.

\bibitem{FWW} J. Fu, Z. Wang, D. Wu, Semilinear equations, the $\gamma_k$ function, and generalized
Gauduchon metrics, \emph{J. Eur. Math. Soc.} {\bf 15} (2013), 659--680.

\bibitem{Gau} P. Gauduchon, La $1$-forme de torsion d'une vari\'et\'e hermitienne compacte,
\emph{Math. Ann.} {\bf 267} (1984), 495--518.

\bibitem{IP01} S. Ivanov, G. Papadopoulos, Vanishing theorems and string backgrounds, \emph{Classical
Quantum Grav.} {\bf 18} (2001), 1089--1110.

\bibitem{IP} S. Ivanov, G. Papadopoulos,
Vanishing theorems on $(l \vert k)$-strong K\"ahler manifolds with torsion,
\emph{Adv. Math.} {\bf 237} (2013), 147--164.

\bibitem{JY} J. Jost, S.-T. Yau, A non-linear elliptic system for maps from Hermitian to Riemannian
manifolds and rigidity theorems in Hermitian geometry, \emph{Acta Math.} {\bf 170} (1993),
221--254; Corrigendum: \emph{Acta Math.} {\bf 173} (1994), 307.

\bibitem{MT} K. Matsuo, T. Takahashi, On compact astheno-K\"ahler manifolds,
\emph{Colloq. Math.} {\bf 89} (2001), 213--221.

\bibitem{Mi} M.L. Michelsohn, On the existence of special metrics
in complex geometry, \emph{Acta Math.} {\bf 149} (1982), 261--295.

\bibitem{STW} G. Sz\'ekelyhidi, V. Tosatti, B. Weinkove, Gauduchon metrics with prescribed volume form,
arXiv: 1503.04491v1 [math.DG].

\end{thebibliography}
\end{document}